\documentclass[11pt]{amsart}

\usepackage{amsmath,amssymb,dsfont,mathtools,xcolor}
\usepackage{graphicx}
\usepackage[colorlinks]{hyperref}
\usepackage{accents}
\usepackage{floatrow}
\newfloatcommand{capbtabbox}{table}[][\FBwidth]
\usepackage{epsfig}
\usepackage{latexsym}
\usepackage{euscript}
\usepackage{lipsum}
\usepackage{xpatch}

\RequirePackage{amsthm,amsmath}   
\newcounter{pulse}[section]
\numberwithin{pulse}{section}  

\newcommand{\thf}{\sc} 

\theoremstyle{plain}
\newtheorem{theorem}[pulse]{\thf Theorem}

\newtheorem{proposition}[pulse]{\thf Proposition}

\newtheorem{claim}[pulse]{\thf Claim}
\newtheorem{corollary}[pulse]{\thf Corollary}
\theoremstyle{definition}
\newtheorem{definition}[pulse]{\thf Definition}
\newtheorem{notation}[pulse]{\thf Notation}

\theoremstyle{remark}
\newtheorem{remark}[pulse]{\thf Remark}

\newcommand{\M}{{\rm{M}}}
\newcommand{\GL}{{\rm GL}}
\newcommand{\On}{{\rm O}_n}
\newcommand{\Dn}{{\rm D}_n}
\newcommand{\Tn}{{\rm T}_n}
\newcommand\SL{{\rm SL}}
\newcommand{\RR}{\mathbb{R}}
\newcommand\ZZ{{\mathbb Z}}
\newcommand{\Fo}{F_{\rm o}}

\newcommand{\SOn}{{\rm SO}_n}
\newcommand{\ke}{black}
\newcommand{\me}{black}

\title{Discrete frames for $L^2(\RR^{n^2})$ arising from tiling systems on $\GL_n({\mathbb R})$}

\author[Mahya Ghandehari]{Mahya Ghandehari}
\address{Department of Mathematical Sciences, University of Delaware}
\email{mahya@udel.edu}

\author[Kris Hollingsworth]{Kris Hollingsworth}
\address{School of Mathematics, University of Minnesota}
\email{kgh@umn.edu}

\renewcommand\paragraph[1]{\par\vspace{1em}\noindent\textbf{#1}}

\begin{document}

\begin{abstract}
A discrete frame for $L^2(\mathbb R^d)$ is a countable sequence $\{e_j\}_{j\in J}$ in $L^2(\mathbb R^d)$ together with real constants $0<A\leq B< \infty$ such that
\[
  A\|f\|_2^2 \leq \sum_{j\in J}|\langle f,e_j \rangle |^2 \leq B\|f\|_2^2,
\]
for all $f\in L^2(\mathbb{R}^d)$.
We present a method of sampling continuous frames, which arise from square-integrable representations of affine-type groups, to create discrete frames for high-dimensional signals.
Our method relies on partitioning the ambient space by  using  a suitable ``tiling system".
We provide all relevant details for constructions in the case of  $\M_n(\RR)\rtimes \GL_n(\RR)$, although the methods discussed here are general and could be adapted to some other settings.
Finally, we prove significantly improved frame bounds over the previously known construction for the case of $n=2$.

\end{abstract}
\maketitle

\noindent
{\sc Keywords:} Continuous Wavelet Transform, Square-Integrable Representation, Discrete Frame,

Quasiregular Representation, Tiling System




\section{Introduction}
\label{sec:introduction}

An important and challenging problem in frame theory is to construct frames which allow suitable representations for signals from various classes of interest.
Frames generalize the notion of a basis for a Hilbert space---that is, frames permit redundant encoding of signals while still yielding stable reconstruction methods.
Since they are explicitly constructed for specific classes of data at hand, generalizing the known frame constructions to other datasets is a very difficult task.
For instance, the classical Fourier basis efficiently represents a harmonic signal in terms of its frequencies (global property of the signal); however, it represents localized signals (i.e. functions with compact support) quite inefficiently.
A fruitful strategy for creating novel frames is to employ the theory of wavelets, which emerged over three decades ago. 
Classical 1-dimensional wavelets may be formed in such a way as to have rapid decay in both spatial and frequency domains, and thus may be used to construct frames for the space of localized 1-dimensional signals.
Unfortunately, the classical wavelet transform is ineffective when dealing with 2-dimensional signals.
In fact, this transform is isotropic and cannot provide information about the geometry of nonisotropic signals in two dimensions.
However, there are instances of 2-dimensional wavelet-type transforms, such as curvelet transforms (introduced in 2000, see \cite{curvelet2002, curvelet2000,  curvelet2003, cts-curvelet-I,cts-curvelet-II}) and shearlet transforms (introduced in 2006, see \cite{shearlet2006,shearlet2,discrete-shearlet}), which have proven extremely successful in the analysis and denoising of signals presenting anisotropic features. 

As the dimension of the signal space grows, developing suitable and efficient frames becomes highly nontrivial.
The easiest method for handling multi-dimensional cases would be to simply use the tensor product of  1-dimensional solutions. However,  this approach does not capture many of the geometric features of high-dimensional signals.
The methods in this article are based on a representation theoretic point of view, developed in \cite{BeT}, where a general framework for the construction of higher dimensional continuous wavelet transforms was investigated.
Essentially, if  a locally compact group $H$ acts on $\mathbb{R}^n$ in such a manner that  $H$ acts freely and transitively on an open subset ${\mathcal O}$ in $\widehat{{\RR^n}}$, then an associated continuous wavelet transform theory can be developed as described in Section \ref{sec:notation}.
The 2-dimensional continuous shearlet transform can be viewed in this manner (see \cite{ScT} and \cite{KuL}).
Various methods, such as careful geometric techniques, can then be used to discretize continuous wavelet transforms, or continuous frames resulting from them, in order to produce discrete frames.

\paragraph{Discrete Frames.}
%
Discrete frames were initially introduced in 1952 by Duffin and Schaeffer \cite{DuS}, but it was not until the mid 1980s into the early 1990s that their use became of increased interest due to the work of Daubechies and her collaborators \cite{DaubechiesGrossmanMeyer1986, Daubechies1990}.
Formally, a \emph{discrete frame} for a separable Hilbert space $\mathcal{H}$ is a set of vectors $\{\phi_x\}_{x\in X}$ with $X$ a countable index set such that for every vector $f\in \mathcal{H}$,%
\[ 
  A\|f\|_{\mathcal{H}}^2
    \leq \sum_{x\in X} |\langle f, \phi_x \rangle_{\mathcal H} |^2
    \leq B\|f\|_{\mathcal{H}}^2,
\]%
for some positive real numbers $A$ and $B$.
If these inequalities are satisfied, the vectors form a stable, possibly redundant system which still allows for reconstruction from the frame coefficients $\{\langle f, \phi_x \rangle\}_{x\in X}$.
Except for the special case of \emph{tight frames} (when $A=B$), reconstruction requires iterative schemes for which the convergence rate is highly sensitive to the frame \emph{condition number}, defined as the ratio $\frac{B}{A}$;
consequently, one goal in designing frames for real applications is to minimize this ratio.
For a detailed introduction to frame theory, see \cite{FramesMemoir2000}.

\paragraph{Continuous Frames.}
A \emph{continuous frame} for a separable Hilbert space $\mathcal{H}$ is a collection of vectors $\{\phi_x\}_{x\in X}$ with $X$ a locally compact Hausdorff space equipped with a positive Radon measure $\mu$ satisfying
\[ 
  A\|f\|_{\mathcal{H}}^2
    \leq \int_X |\langle f, \phi_x \rangle_{\mathcal H}|^2 \, d\mu(x) 
    \leq B\|f\|_{\mathcal{H}}^2 
    \qquad \forall f\in \mathcal{H},
\]%
for some positive real numbers $A$ and $B$.
Notice that when $\mu$ is the counting measure on a countable space $X$, this agrees with the previous definition of a discrete frame.

The term continuous frame appears to be attributable to Ali, Antoine, and Gazeau in \cite{AliAntoineGazeau1993}, although the concept did not originate with them.
Continuous frames for $L^2(\RR^n)$ were known to Calder\'on in the 1960s (see \cite{Calderon1964})---as such, some authors refer to the following as the \emph{Calder\'on reproducing formula} (for example, \cite{Laugesen2002}).
Essentially, when $\{\phi_x\}_{x\in X}$ is a continuous frame, there exist vectors $\{\tilde{\phi}_x\}_{x\in X}$, called the \emph{dual frame}, such that%
\[ 
  f 
    = \int_X \langle f, \phi_x\rangle \tilde{\phi_x} \, d\mu(x), 
    \qquad \forall f\in \mathcal{H}
\]%
where the integral is interpreted in the weak sense.
A thorough introduction for continuous reproducing formulas can be found in section 2 of \cite{Fornasier2005}.

A \emph{wavelet frame} is one in which the frame vectors are all formed from translations and dilations (or translations and modulations in the case of Gabor frames) of a single \emph{mother} or \emph{analyzing wavelet}.
Generalized versions of continuous wavelet transforms and frames, defined in Section \ref{sec:notation},  have been investigated extensively in the past couple of decades.
In the present work, we restrict ourselves to constructions on Hilbert spaces, but information on continuous frames in certain Banach spaces can be found in \cite{Fornasier2005}.
For examples of wavelet transforms based on the theory of square-integrable representations, see \cite{shearlet2006, Currey-Fuhr-Taylor, Twareque2014, Namngam-Schulz, Cordero-Tabacco, MacArthur-Taylor},
 and for the general theory governing their construction and behavior see \cite{Cordero-etal, Fuhr-2010, Fuhr-Mayer, Fuhr-wavefront,Fuhr-coorbit, Fuhr-book}.

\paragraph{Discretization.}
{\color{\me}
In this article, we focus on the problem of constructing discrete  frames by sampling continuous wavelet frames in the orbit of a square-integrable irreducible representation.
In 1980's, H.G. Feichtinger and K.H. Gr\"{o}chenig developed \emph{coorbit theory} \cite{FG88, CoorbitAtomicDecomposition:FG, FG89}, which 
has become a powerful tool for producing atomic decomposition for a variety of Banach spaces (of functions or distributions). 
It is notable that many important classical function spaces in harmonic analysis are instances of coorbit spaces. This remarkable fact motivates the study of coorbit spaces, as a means for developing unified methods to study various function spaces. 
The coorbit theory starts with  an integrable representation $\pi$ of a locally compact group $G$ on a Hilbert space ${\mathcal H}_\pi$; coorbit spaces are then defined as certain Banach spaces related to $\pi$. Luckily, continuous wavelet transforms defined using $\pi$ (and suitable analyzing wavelets $\xi\in {\mathcal H}_\pi$) extend to every coorbit space. This crucial feature enables one to view coorbit spaces as certain subspaces of ``reasonable'' Banach function spaces over the ambient group $G$.

Atomic decompositions of a coorbit space can be obtained by means of discretizing certain convolution operators on the associated Banach function space. This approach leads to results of the following type: if the analyzing vector $\xi$ corresponds to a function of a suitable Weiner-type space, then for any sufficiently dense well-spread subset $\{x_i\}_{i\in I}$ of $G$, vectors in the coorbit space can be decomposed into atoms of the form $\{\pi(x_i)\xi\}_{i\in I}$. Such decompositions are usually referred to as \emph{atomic decompositions} in mathematics or discrete expansions with respect to \emph{coherent states} in physics.
Atomic decompositions were used in  \cite{DescribingFunctions:G} to construct discrete Banach frames in the orbit of an integrable irreducible representation. 
}
The condition of  integrability was replaced by the weaker assumption of square-integrability in  \cite{SamplingTheorem:FG}.
A more general answer to the discretization problem is obtained in \cite{freeman:2016:DiscretizationProblem}, where the authors prove that every bounded continuous frame may be sampled to obtain a discrete frame.
Building on this work, F\"{u}hr and Oussa, \cite{fuhr:2018:GroupsWFramesOfTranslates}, found large classes of Lie groups $G$ for which $L^2(G)$ admits discrete frames of translates.
Additionally, they proved several necessary and some sufficient conditions for existence of such discrete frames.
In contrast,  it was shown in \cite{CHRISTENSEN1999292} that $L^2(\RR^n)$ does not admit discrete frames of pure translates, in the sense that no collection $\bigcup_{k=1}^r\{g_k(x-a)\}_{a\in \Gamma}$
can form a frame for $L^2(\RR^n)$.  This demonstrates the necessity of more complex methods, such as the one we explore in the current work.

None of the above mentioned results are explicitly constructive in nature.
{\color{\ke}%
In fact, computing an explicit sufficiently dense well-spread subset to guarantee atomic decompositions in \cite{CoorbitAtomicDecomposition:FG} is a very difficult task.}
Examples of previous approaches for discretizing continuous wavelet transforms to construct discrete frames may be found in \cite{Hemmat2017, Maggioni2004, NonUniformPainlessDecomps:CMR}.
This question has also been investigated in the study of quantum mechanics as a special case of coherent states, see, for example, chapter 9 of \cite{AliAntoineGazeau2014book} and references therein.

\paragraph{Contents.}
In the present article, we restrict our attention to $L^2(\RR^{n^2})$, as we seek to improve the results of \cite{GST},
wherein the theory of square-integrable representations was combined with the geometry of the Euclidean space to construct discrete frames for $L^2({\mathbb R}^4)$.
{\color{\ke}%
The current article exhibits a prototypical example of methods for constructing concrete (and well-structured) frames; these methods need to be tailored to the geometry of the group at hand. Once a group with a suitable underlying geometric structure is chosen, we can take advantage of ideas such as tiling systems to produce frames, whereas frame constructions  in coorbit theory are based on the notions of well-spread sets and partitions of unity. For a detailed discussion, see Remark \ref{refremark}.}
As in \cite{GST}, we follow the general method developed in \cite{BeT} to construct a tight continuous frame using representation theory of the group $\M_n(\RR) \rtimes \GL_n(\RR)$.  We then obtain a discrete frame through careful geometric techniques for discretizing the reproducing formula.
As a result, we improve the construction in \cite{GST} significantly by reducing the frame condition number from about 1782 to 33 for $L^2({\mathbb R}^4)$.
More importantly, the previous construction was provided only for $L^2(\RR^4)$ which we generalize to $L^2(\RR^{n^2})$ for any $n\in \mathbb{N}$.
Note that square-integrable representations provide us with the only reasonable framework to produce discrete frames, since the existence of a discrete frame in this setting implies the square-integrability of the unitary representation generating the associated continuous frame (see \cite{Aniello2001} for more details).

Finally, we compare our work with that of Heinlein and his collaborators in \cite{Heinlein2003,HeinleinDrexl2003}.
The two approaches rely on the same representation theoretic viewpoint. Indeed, we both discretize continuous wavelet frames obtained from a square-integrable representation of the affine group; however, we note three key differences.
First, their strategy is to use integrated wavelets, which are averages of a wavelet in the Fourier domain over any countable partition of the space (in our case $\RR^{n^2}$) into compact sets. 
With this method, they obtain tight frames, but this comes at the cost that the analysis and reconstruction filters for $n$ resolution scales require $(n+1)$ Fourier transforms.
Consequently, the computational cost of this method on higher-dimensional data would be extremely high, whereas  our method only requires one Fourier inversion of the analyzing wavelet.
Second, we obtain and work with a much more structured  decomposition than their general partition (which they call detail decomposition). 
Our approach starts with a ``tile'' which under the action of a discrete, countable set covers the space.
The structure of our tiling system reduces the amount of information needed for implementation, as one only needs to know the ``mother tile'' and the form of the discrete set acting on it. Thus, the challenge here is to obtain such suitable tiling systems, which we overcome by carefully investigating  the geometry of the space.
Last, integrated wavelets require a two-step discretization.
That is, they must discretize translations and dilations separately, and it is the intermediate discretization on dilations alone which provides a tight frame.
Our method performs the discretization simultaneously, allowing for a one-step process.

\paragraph{Organization.}
This paper is organized as follows.
In Section \ref{sec:notation} we collate all necessary definitions, notations, and background. %
In Section \ref{sec:construction} we provide the general theory for constructing discrete frames from the continuous wavelet transform by means of tiling systems. %
In Section \ref{sec:GeneralTilingSystem} we provide a general construction for tiling of $\GL_n(\RR)$ for any $n\in \mathbb{N}$. %
In Section \ref{sec:ComputationAndExample} we compute the frame bounds for the tiling system and derive an upper bound on the frame condition number as a function of the dimension $n$.
We conclude the section by providing explicit details of the concrete construction for the case when $n=2$. %
In section \ref{sec:Discussion} we end with a brief discussion of the future work.

\section{Notations and Definitions}\label{sec:notation}

Let $n\in {\mathbb N}$, and $\M_n(\mathbb{R})$ denote the set of $n\times n$ real matrices. Equipped with matrix addition and  the topology of $\mathbb{R}^{n^2}$, the set $\M_n(\RR)$ can be viewed as a locally compact abelian group. Let $\GL_n(\mathbb{R})$ denote the subset of $\M_n(\RR)$ containing all  $n\times n$ real matrices with nonzero determinant. It is well-known that $\GL_n(\RR)$ is an open subset of $\M_n(\RR)$, as the determinant  is a continuous function (in fact a polynomial) in the matrix entries. So, $\GL_n(\RR)$ turns into a locally compact group, when equipped with matrix multiplication and the induced topology of $\RR^{n^2}$.  
Elements of $\GL_n(\RR)$ and $\M_n(\RR)$ can be combined to form \emph{affine transformations} as defined below.
\begin{definition}
For $x\in \M_n(\RR)$ and $h\in \GL_n(\RR)$, let $[x,h]$  denote the affine transformation of $\M_n(\RR)$ given by 
\[
  [x,h]y=hy+x\quad \mbox{ for } y\in \M_n(\RR).
\]
\end{definition}
Let $\M_n(\RR)\rtimes \GL_n(\RR)=\left\lbrace [x,h]:  x\in \M_n(\RR), h\in \GL_n(\RR)\right\rbrace$ denote the collection of all affine transformations  defined above. 
Composition of transformations can be seen as the following product operation.
\begin{equation}\label{1}
[x_1, h_1][x_2, h_2]=[x_1+h_1x_2, h_1h_2].
\end{equation}
Then $G_n:=\M_n(\RR)\rtimes \GL_n(\RR)$, together with product (\ref{1}), forms a non-abelian locally compact group when given the product topology. Let $I_n$ denote the $n\times n$ identity matrix, and $0_n$ denote the $n\times n$ zero matrix. It is very easy to check that 
$[0_n, I_n]$ is the identity of $G_n$, and $[x,h]^{-1}=[-h^{-1}x, h^{-1}]$ for $[x,h]\in G_n$. 

\paragraph{Haar Integration.} The most useful measure on  $G_n$ is its (unique, up to scaling) left-invariant measure, called the left Haar measure, which we explicitly describe. In what follows, all the functions appearing in the integration formulas are integrable and defined on the appropriate domains. 
First, we equip $\M_n(\RR)$ with Lebesgue measure under the identification with $\mathbb{R}^{n^2}$,  and let \hspace{3mm}$\int\limits_{\mathclap{\M_n(\RR)}}f(x)\,dx$  denote the Lebesgue integration. That is, 
\[
dx=dx_{11}dx_{12}\cdots dx_{1n}dx_{21}\cdots dx_{n1}\cdots dx_{nn}\ \mbox{ if } x=\begin{pmatrix}
x_{11}& x_{12} &\cdots & x_{1n}\\
\vdots& \vdots &\ddots & \vdots\\
x_{n1}& x_{n2} &\cdots & x_{nn}
\end{pmatrix}.
\]
For $\GL_n(\mathbb{R})$, the left Haar integration is given by \hspace{3mm}$\int\limits_{\mathclap{\GL_n(\mathbb{R})}}g(h)\,dh$, where
\begin{equation*}\label{3}
dh=\frac{dh_{11}dh_{12}\cdots dh_{1n}dh_{21}\cdots dh_{n1}\cdots dh_{nn}}{|\det(h)|^n} \quad \mbox{ if }h=\begin{pmatrix}
h_{11}& h_{12} &\cdots & h_{1n}\\
\vdots& \vdots &\ddots & \vdots\\
h_{n1}& h_{n2} &\cdots & h_{nn}
\end{pmatrix},
\end{equation*}
where $h=[h_{ij}]_{i,j=1}^n$ is a generic element of $\GL_n(\RR)$.
It turns out that the left Haar measure on $\GL_n(\mathbb{R})$ is also right invariant, i.e. $\GL_n(\RR)$ is unimodular.
So, for any $h'\in \GL_n(\mathbb{R})$ and a compactly supported function $g$,
$$\int_{\GL_n(\mathbb{R})}g(h'h)\;dh=\int_{\GL_n(\mathbb{R})}g(hh')\;dh=\int_{\GL_n(\mathbb{R})}g(h^{-1})\;dh=\int_{\GL_n(\mathbb{R})}g(h)\;dh.$$
Now we can describe left Haar integration on $G_n$. For a compactly supported function $f:G_n\rightarrow {\mathbb C}$,
\begin{equation}\label{4}
\int_{G_n}f([x, h])\;d[x, h]=\int_{\GL_n(\mathbb{R})}\int_{\M_n(\mathbb{R})}f([x, h])\, \frac{dx\;dh}{|\det(h)|^{n}}.
\end{equation}
It is a routine calculation to show that the integration defined in (\ref{4}) is invariant under left translations.

This integration is not right invariant. However, we have the following formula for handling the case of right translation.
$$\int_{G_n}f([x, h][y,k])\;d[x, h]=|\det(k)|^n\int_{G_n}f([x, h])\;d[x, h],$$
 for every $[y,k]\in G_n$.  See \cite{HeR} or \cite{Fol}
for the properties of Haar measure and the Haar integral in general.
 
\paragraph{Transferring to Fourier Domain.} 
It turns out that the existence of a continuous wavelet transform depends crucially on the geometric features of the underlying group, and in particular the geometry of the action of $\GL_n(\RR)$ on the Pontryagin dual of $\M_n(\RR)$. 
To make notation more clear, $\M_n(\RR)$ is denoted by $A$, when it is identified with $\RR^{n^2}$ as an abelian group. Let $\widehat{A}$ denote the dual (i.e. the group of characters) of $\M_n(\RR)$ under this identification. 
For the purpose of Fourier analysis on $\mathbb{R}^{\mathrlap{n^2}}\,\;,$ identified with $A$, we choose the
following way to pair $A$ with $\widehat{A}$. 
 For $b=[b_{ij}]_{i,j=1}^n\in A$, define
$\chi_{b} \in \widehat{A}$ by
\begin{equation}\label{5}
\chi_b(x)=e^{2\pi i\text{tr}(bx)}, ~~ \text{for}~~  x\in A.
\end{equation}
We have, $\widehat{A}=\left\lbrace \chi_b: ~~ b\in A\right\rbrace$. Thus, $\widehat{A}$
can also be identified with $\mathbb{R}^{n^2}$, and Haar integration on $\widehat{A}$
is simply the Lebesgue integral,  i.e.
$$\int_{\widehat{A}}g(\chi)\;d\chi=\int_{\mathbb{R}^{n^{2}}}g(\chi_b)\;db=
\int_{\mathbb{R}}\cdots\int_{\mathbb{R}}g(\chi_{(b_{11},\cdots , b_{nn})})
\;db_{11}\cdots db_{nn}.$$
For $f\in L^1(A)$, the Fourier transform
$\widehat{f}\colon \widehat{A}\to\mathbb{C}$ is given by
$\widehat{f}(\chi)=\int_Af(x)\overline{\chi(x)}dx,$ for all $\chi\in \widehat{A}$.
If $f\in L^1(A)\cap L^2(A)$, then $\widehat{f}\in L^2(\widehat{A})$ and
$\|\widehat{f}\|_2=\|f\|_2$. So, the Fourier transform extends to a unitary map
$\mathcal{P}\colon L^2(A)\to L^2(\widehat{A})$, the Plancherel transform,
such that $\mathcal{P}f=\widehat{f}$, for all $f\in L^1(A)\cap L^2(A)$.

The group $\GL_n(\RR)$ acts on  $A$ by matrix multiplication. This action determines an action of $\GL_n(\RR)$ on the dual space $\widehat{A}$ by
 $h\cdot\chi_b=\chi_{bh^{-1}}$, for $b\in A$ and $h\in \GL_n(\RR)$.
This action scales Lebesgue measure, so that, for any integrable function $\xi$ on $\widehat{A},$
\begin{equation}\label{6}
\int_{\widehat{A}}\xi(\chi)\;d\chi=|\det(h)|^{-n}\int_{\widehat{A}}\xi(h\cdot \chi)\;d\chi.
\end{equation}
%
\paragraph{A Square-Integrable Irreducible Representation.}
We now give two equivalent forms of the quasi-regular representation of $G_n$, which is the irreducible 
representation that has been used in  \cite{GST} to construct continuous wavelet transforms. 
Let ${\mathcal H}$ be a Hilbert space, and ${\mathcal U}({\mathcal H})$ denote the group of unitary operators on ${\mathcal H}$.
A \emph{continuous unitary representation} of  a locally compact group $G$ on ${\mathcal H}$ is a group homomorphism $\pi:G\rightarrow {\mathcal U}({\mathcal H})$ which is WOT-continuous, i.e. for every vectors $\xi$ and $\eta$ in ${\mathcal H}$, the function
$$\pi_{\xi,\eta}:G\rightarrow {\mathbb C}, \quad g\mapsto \langle\pi(g)\xi,\eta\rangle$$
is continuous.
Functions of the form $\pi_{\xi,\eta}$, for vectors $\xi$ and $\eta$ in ${\mathcal H}$, are called the \emph{coefficient functions} of $G$ associated with the representation $\pi$. 
If $\sigma_1$ and $\sigma_2$ are two representations of $G$ on $\mathcal{H}_1$ and $\mathcal{H}_2$ respectively, we say $\sigma_1$ is (unitarily) equivalent to $\sigma_2$ if there exists a unitary transformation $U\colon \mathcal{H}_1\to \mathcal{H}_2$  such that %
\[
  U\sigma_1(x)=\sigma_2(x)U, ~ ~~ ~ \text{for all}~~~~  x\in G.
\] 
A representation $\pi$ is called irreducible if $\left\{ 0\right\} $ and $\mathcal{H}$ are the only $\pi$-invariant closed subspaces of $\mathcal{H}$. 
Let $\widehat{G}$ denote the space of equivalence classes of irreducible representations of $G$. 
An introduction to the representation theory of locally compact groups can be found in \cite{Fol}.

An irreducible representation $\pi:G\rightarrow {\mathcal U}(\mathcal{H})$ is called \emph{square-integrable} if there exist nonzero $\xi, \eta\in \mathcal{H}$ such that $\pi_{\xi,\eta}\in L^2(G)$, where $L^2(G)$ is the collection of all square-integrable complex-valued functions on $G$.
With $\xi\in \mathcal{H}\setminus\left\lbrace 0\right\rbrace $ fixed, if there exists one nonzero $\eta^{\prime}\in \mathcal{H}$ with $\pi_{\xi, \eta^{\prime}}\in L^2(G),$ then $\pi_{\xi, \eta}\in L^2(G)$ for any $\eta\in \mathcal{H}.$ Such a vector $\xi$ is called \emph{admissible} and the set of admissible vectors is dense in $\mathcal{H}$.
For a comprehensive discussion of continuous wavelet transforms and their connection with square-integrable representations, see \cite{Fuhr-book}.

%
The results in the following proposition are proved in \cite{GST}, so we skip the proof. 
\begin{proposition}
 Let  ${\mathcal H}_1=L^2(A)$ and ${\mathcal H}_2=L^2(\widehat{A})$. Then,
 \begin{itemize}
\item[(i)] $\rho$ is a unitary representation, where
\begin{equation*}\label{7}
\rho:G_n\rightarrow {\mathcal U}({\mathcal H}_1), \quad \rho[x, h]g(y)=|\det(h)|^{-n/2}g(h^{-1}(y-x)), 
\end{equation*}
for all  $[x, h]\in G_n$ and $y\in A$.
%
\item[(ii)] $\pi$ is a unitary representation, where
\begin{equation*}\label{8}
\pi:G_n\rightarrow {\mathcal U}({\mathcal H}_2), \quad \pi[x, h]\xi(\chi)=|\det(h)|^{n/2}\,\overline{\chi(x)}\,\xi(h^{-1}\cdot \chi), 
\end{equation*}
for all $[x, h]\in G_n$ and $\chi\in \widehat{A}$.
\item[(iii)] The representations $\rho$ and $\pi$ are equivalent square-integrable irreducible representations of $G_n.$
Namely, we have $\pi[x,h]\mathcal P = \mathcal P \rho[x,h]$, where $\mathcal P$ is the Plancherel transform.
%
\end{itemize}
\end{proposition}

\paragraph{Continuous Wavelet Transforms.}
We now use the square-integrable representations defined in the previous section to construct a continuous wavelet transform. Recall that by $A$ we denote $\M_n(\RR)$, when it is identified with $\RR^{n^2}$ as an abelian group. Let $\rho$ be the square-integrable representation defined earlier. 
\begin{definition}\label{def:wavelet-trans}
A function $\psi\in L^2(A)$ is called a \emph{wavelet} if
\begin{equation}\label{12}
\int_{\GL_n(\RR)}\left| \,\widehat{\psi}(\chi_h)\,\right|^2\,dh=1.
\end{equation}
A \emph{continuous wavelet transform} (CWT) associated with a  wavelet $\psi\in L^2(A)$ is  the linear transformation defined as 
$$V_{\psi}\colon L^2(\M_n(\RR))\to L^2(G_n), \quad V_{\psi}f[x,h]=\left\langle f,\rho[x,h]\psi\right\rangle,$$
for $f\in L^2(A), [x,h]\in G_n$.
\end{definition}
It is known that a continuous wavelet transform  $V_{\psi}$ is an isometry of $ L^2(\M_n(\RR))$ into $L^2(G_n)$, that is 
\[
  \langle f,g\rangle
    =\int_{G_n}
      \left\langle f, \rho[x,h]\psi\right\rangle
      \left\langle\rho[x,h]\psi,g\right\rangle
      \, d[x,h],
\]
for any $f,g\in L^2(\M_n(\RR))$. This formula can be written in the following weak integral form:
\begin{equation}\label{eqn:General Reconstruction Formula}
f=\int_{G_n}\left\langle f,\rho[x,h]\psi\right\rangle\rho[x,h]\psi\; d[x,h],
\end{equation}
for any $f\in L^2(A)$.  For notational convenience, we denote $\rho[x,h]\psi$ by $\psi_{x,h}$. 
See \cite{Fuhr-book}, for a detailed discussion about general CWTs,  and \cite{GST} for the details of the above mentioned wavelet transform. 

\section{Constructing Frames using CWTs}\label{sec:construction}
In this section, we present a summary of the results from \cite{BeT}, explaining how CWTs may be used to construct discrete frames.
The objective here is to construct frames of the form $\{\rho[x,h]\psi\}_{[x,h]\in P}$, where $P$ is a discrete subset of $G_n$, and $\rho$ and $\psi$ are as in Definition \ref{def:wavelet-trans}. The frame constructions in \cite{BeT} heavily rely on the concept of a ``tiling system'' which we define here. 

In \cite{GST}, the first author and collaborators use the approach of \cite{BeT}  to construct a discrete frame for $L^2(\RR^4)$. 
More precisely,  they obtain a suitable tiling system for $\GL_2(\RR)$, which they use to discretize the continuous wavelet transform on $\RR^4$.  In this section, we extend their frame construction to  general $n$, obtaining discrete frames for $L^2(\RR^{n^2})$. 
In addition, we obtain a much tighter gap between the frame bounds. 

For what follows, we restrict ourselves to the groups which are being studied in this paper; although our methods can be carried out in more general settings.

\begin{definition}\label{def:TilingSystem}
Let $P$ be a countable subset of $\GL_n({\mathbb R})$, and $F$ be an open relatively compact subset
 of $\GL_n(\RR)$. The pair $(F,P)$ is called a \emph{tiling system} for $\GL_n(\RR)$ if the following
two conditions are satisfied:
\begin{itemize}
\item[(i)] $\lambda_{\GL_n}(\overline{F}\cdot p \bigcap \overline{F}\cdot q)=0$ for every distinct pair $p,q\in P$,
\item[(ii)] $\lambda_{\GL_n}\left(\GL_n(\RR)\setminus\bigcup_{p\in P}\overline{F}\cdot p \right)=0$,
\end{itemize}
where $\lambda_{\GL_n}$ denotes the left Haar measure of $\GL_n(\RR)$, and $\overline{F}$ denotes the closure of $F$ in $\GL_n(\RR)$.
Note that conditions (i) and (ii) remain unchanged if one replaces $\lambda_{\GL_n}$ with $\lambda_{\M_n}$, the left Haar measure on $\RR^{n^2}$.
\end{definition}
\begin{remark}
In the above definition, we think of a subset $S$ of $A=\M_n(\RR)$ as a subset $\widetilde{S}$ of $\widehat{A}$ in the following natural manner: $\widetilde{S}=\{\chi_b:b\in S\}$. With this identification in mind, we have 
$\widetilde{\overline{F}\cdot p}=p^{-1}\cdot \widetilde{\overline{F}}$.
\end{remark}
\begin{remark}\label{rem:move-to-R}
Recall that we identify elements of  $\M_n(\RR)$ with vectors in $\RR^{n^2}$. We do so via the map $\Phi$ given by 
\begin{equation}\label{eq:Phi}
\left(\begin{array}{ccc}
x_{11} & \ldots& x_{1n} \\
\vdots&\ddots&\vdots\\
x_{n1}   &\ldots&  x_{nn}
\end{array}
\right) \xmapsto[]{\;\Phi\;} \left( x_{11},\ldots, x_{1n},\ldots,x_{n1},\ldots, x_{nn}\right).
\end{equation}
The $n\times n$ matrix structure is only used when multiplication is involved. 
This identification allows us to transfer the notion of tiling to $\RR^{n^2}$. Let $(F,P)$ be a tiling system for $\GL_n(\RR)$.
For each $p\in P$, let $L^2(\Phi(\overline{F}\cdot p))$ denote the closed subspace of
$L^2(\RR^{n^2})$ consisting of functions that are zero almost everywhere on
$\RR^{n^2}\setminus\Phi(\overline{F}\cdot p)$.
Noting that $\GL_n(\RR)$ is a co-null subset of $\M_n(\RR)$ (and thus $\RR^{n^2}$), we have that
\[
L^2\left(\RR^{n^2}\right)=\bigoplus_{p\in P}L^2(\Phi(\overline{F}\cdot p)).
\]
With this in mind, we can think of $(F,P)$ as a tiling system for $\M_n(\RR)$, or equivalently for $\RR^{n^2}$, as well. 
\end{remark}

We now give a brief review  of the method introduced in \cite{BeT} for constructing a frame from a tiling system. We first introduce some notations, which will be useful for the next theorem.
\begin{notation}\label{notation:cube-J}
Let  $(F,P)$ be a tiling system, and $R$ be a hypercube in $\M_n(\RR)$ whose interior contains $\overline{F}$, i.e. $R$ is defined by real numbers $a_{ij}<b_{ij}$, $1\leq i,j\leq n$ as follows.
$$R=\left\{[x_{ij}]\in \M_n(\RR): a_{ij}\leq x_{ij}\leq b_{ij} \ \mbox{ for }\ i,j=1,\ldots, n\right\}.$$
%
%
Let $|R|=\prod_{i,j=1}^{n}(b_{ij}-a_{ij})$ be the volume of the cube, and  $J=\left\{\left[\frac{m_{ij}}{b_{ji}-a_{ji}}\right]_{ij}\in \M_n(\RR): m_{ij}\in \ZZ\right\}$.  
\end{notation}
Every ${\gamma}\in J$ defines a character on the hypercube $R\subseteq \M_n(\RR)$ as follows. 
$$e_{\gamma}(y)=\frac{1}{\sqrt{|R|}}\mathds{1}_R({y})\exp\left(2\pi i\,  \text{tr}({\gamma}y)\right),\quad \mbox{ for } {y}\in R,$$
where $\mathds{1}_R$ is the characteristic function of $R$. Then $\{e_{{\gamma}}:\ {\gamma}\in J\}$ is an orthonormal basis for $L^2(R)$.

We now prove a slightly different version of Theorem 3 of \cite{BeT}. Even though our proof is similar to theirs, we manage to obtain a much tighter frame condition number due to our new definition for the constant $M$. 
In the following section, we will show that the conditions in our theorem can be met for  $G_n=\M_n(\RR)\rtimes\GL_n(\RR)$ for any $n\in {\mathbb N}$.

\begin{notation}\label{notation:Ip}
Let $(F, P)$ be a tiling system for $\GL_n(\RR)$, and $R$ be a hypercube containing $\overline{F}$. Let $\Fo$ be an open set satisfying $\overline{F}\subseteq \Fo\subseteq R$. For $p\in P$, define $I_{\Fo}(p):=\Big\{k\in P:\ \overline{F}\cdot p \cap \Fo\cdot k\neq \emptyset \Big\}$.
Let $M:=\sup_{p\in P}|I_{\Fo}(p)|$. Observe that $M$ is finite according to Lemma 4 of \cite{BeT}.
\end{notation}

{\color{\me} 
\begin{remark}\label{rem:well-spread} 
Tiling systems are closely related to well-spread sets from coorbit theory; see for example \cite[Definition 3.2]{CoorbitAtomicDecomposition:FG} for the definition of a well-spread set. Indeed, if $(F,P)$ is a tiling system, then $P$ forms an $\Fo$-dense and $F$-separated set. This is an easy consequence of 
the fact that for open sets the conditions ``disjointness'' and ``having an intersection of (Haar) measure 0'' are equivalent.
In a similar vein, the collection of characteristic functions $\{\mathds{1}_{F\cdot p}\}_{p\in P}$ can be thought of as a ``partition of unity'' in coorbit theory, with the adjustment that $\sum_{p\in P}\mathds{1}_{F\cdot p}\equiv 1$ almost everywhere. The reader may notice a continuity condition for the partition of unity in the coorbit theory literature. 
However, this condition is not necessary; see \cite[4.1]{DescribingFunctions:G} for a more general definition of partition of unity. 
\end{remark}
}

\begin{theorem}\label{theorem:frame}
Let $(F,P)$ be a tiling system for $\GL_n(\RR)$, with $R$ and $\Fo$ as in Notation \ref{notation:Ip}.
Let $g\in L^2(\M_n({\mathbb R}))$ be such that $\mathds{1}_{\overline{F}}\leq \widehat{g}\leq \mathds{1}_{\Fo}$.
Then
$\{\rho[\lambda,p]^{-1}g: (\lambda,p)\in J\times P\}$ is a discrete frame
in $L^2(\M_n(\RR))$, where $J$ is defined in Notation \ref{notation:cube-J}.
Moreover, the lower and upper frame bounds are given by%
\[
  |R|\,\|f\|_2^2
  \leq 
  \sum_{k\in P}\sum_{{\gamma}\in J}\left|\langle f,\rho[{\gamma},k]^{-1}g\rangle_{L^2(\M_n(\RR))}\right|^2
  \leq
  M|R| \,\|f\|_2^2,
\] 
with $M$ as in Notation \ref{notation:Ip}.
\end{theorem} 
\begin{proof}
For an arbitrary $f\in L^2(\M_n(\RR))$, we have 
\begin{eqnarray*}
\sum_{k\in P}\sum_{{\gamma}\in J}\left|\langle f, \rho[{\gamma},k]^{-1}g\rangle_{L^2(\M_n(\RR))}\right|^2
&=&|R|\int\limits_{\M_n(\RR)} |\widehat{f}( \chi_{b})|^2 \left (\sum_{k\in P}|\widehat{g}(\chi_{bk^{-1}})|^2 \right) db,
\end{eqnarray*}
where the details follow as in \cite{BeT}, and so we omit them here.
Note that for almost every $b\in \M_n(\RR)$, there exists an element $k\in P$ for which $b\in \overline{F}\cdot k$. This, together with the fact that $\widehat{g}\geq \mathds{1}_{\overline{F}}$, implies that $\sum_{k\in P}|\widehat{g}(\chi_{bk^{-1}})|^2\geq 1$ for almost every $b\in \M_n(\RR)$. Thus, $|R|$ is a lower bound for the frame. 
To obtain the upper bound, we note that $\widehat{g}\leq 1$. Moreover, for almost every $b\in\M_n(\RR)$, there are at most $M$ values of $k\in P$ for which $|\widehat{g}(\chi_{bk^{-1}})|>0$, as $\widehat{g}\leq \mathds{1}_{\Fo}$. This finishes the proof.
\end{proof}

{\color{\ke}
\noindent
\begin{remark}
\label{refremark}
 By Theorem 2.9 in \cite{Fuhr-coorbit}, our construction in Theorem \ref{theorem:frame} produces concrete examples of (Banach) frames in the sense of coorbit theory,  provided that $\widehat g$ is chosen to be a bandlimited Schwartz function.
Namely,  standard results of coorbit theory (cf. \cite{CoorbitAtomicDecomposition:FG,FG89}) can  be applied to such a function $g$ to show that the proposed discretization of Equation \eqref{eqn:General Reconstruction Formula} gives rise not only to a frame for $L^2(\mathbb R^{n^2})$, but also atomic decompositions for an entire family of coorbit spaces (converging in the corresponding coorbit space norms), with explicit control over the frame condition number.
 Here, we note some key differences and similarities between our construction and those obtained using the general theory of coorbits.
\begin{itemize}
\item[(i)]
We find the explicit sampling set $P\times J$ through careful and direct consideration of the group $G_n$. 
This construction results in instances of frames that cannot be obtained from the general coorbit theory, since our analyzing vector $g$ in Theorem \ref{theorem:frame} does not need to satisfy conditions of ``better vectors'' (i.e. vectors  corresponding to a certain Weiner-type space) from coorbit theory. 
As we do not impose any smoothness criteria on $g$, we have more freedom in choosing a mother wavelet.
In fact, any function satisfying the support condition $\mathds 1_F\leq \widehat g \leq \mathds 1_{\Fo}$ can serve as an analyzing wavelet for our construction.
One noteworthy example is the case $\widehat g=\mathds 1_F$, for which the proof of Theorem \ref{theorem:frame} shows our construction results in a tight discrete frame.
  
%
\item[(ii)]
Specific properties of the sampling set $P\times J$ enable us to provide explicit frame bounds in our proof.
If $\widehat g$ is a bandlimited Schwartz function, 
results of coorbit theory guarantee that for small enough $U$, the set $P\times J$ may be replaced with any well-spread $U$-dense sampling set.
However, the general theory provides conservative estimates on the required density; see for example conditions on $U$ in Theorem 5.3 and 5.5 of \cite{DescribingFunctions:G}.
This can make finding explicit constructions of sampling sets a challenging task, even for particular groups such as those studied in this article.
Moreover, coorbit theory would not provide the explicit frame bounds calculated in Section \ref{sec:ComputationAndExample}.
\item[(iii)] In \cite{1905.04934}, improved density criteria and frame bound estimates (arising from coorbit theory) for affinely generated systems are obtained. These results again rely on the analyzing wavelet being a ``better'' vector in the sense of coorbit theory. In the special case of a bandlimited Schwartz analyzing wavelet, our results overlap with those of \cite{1905.04934}, except our frame bounds are calculated explicitly in terms of the tiling system.
\end{itemize}
\end{remark}}
%
\section{Tiling System for \texorpdfstring{$\GL_n(\RR)$}{GLn(R)}}\label{sec:GeneralTilingSystem}
In this section, we generalize the construction of a tiling system for $\GL_2(\RR)$ given in \cite{GST} to a tiling system  for $\GL_n(\RR)$. We then show this construction meets
the conditions of Theorem \ref{theorem:frame}. We compute the corresponding frame bounds in Section \ref{sec:ComputationAndExample}.

The definition of our tiling system in this section and the computations thereafter are inspired by the Iwasawa decomposition for $\GL_n(\RR)$.
For matrices, this decomposition is equivalent to the well-known Gram decomposition of a matrix;
with the upper triangular part further decomposed into a diagonal matrix with positive entries
and a unit upper triangular matrix.
Finally, we factor out the determinant to use as a parameter in our tiling system.
While this realization works for matrix groups, the Iwasawa decomposition is much more general, and could be applied to other semisimple Lie groups.
This  decomposition can be viewed as a change of variables
when computing the Haar measure, for which we provide the relevant formulas here as well.

\begin{theorem}[{\cite[Proposition 2.3]{Spherical:Jorgenson}}]
\label{thm:IwasawaHaarMeas}
Each $a\in \GL_n(\RR)$ can be uniquely decomposed as $a=skwy$. Here, $k\in \On$, $s\in \RR^{+}$, $w\in \Dn$, and $y\in \Tn$; where $\On$ is  the orthogonal group in dimension $n$,  $\Dn$ is the group of diagonal matrices with positive diagonal entries and determinant 1, and $\Tn$ is the group of unit upper triangular matrices. That is%
\[
\GL_n(\RR) = 
\left\{
s
k
\begin{pmatrix}
  w_1 & & & \\
    & \ddots & & \\
    & & w_{n-1} & \\
    & & & \prod_{i=1}^{n-1} w_i^{-1}
\end{pmatrix}
\begin{pmatrix}
  1 & y_{1,2} & \cdots & y_{1,n} \\
    & \ddots & \ddots & \vdots \\
    & & 1 & y_{n-1,n} \\
    & & & 1
\end{pmatrix}
\,\,\middle|\,\,
\begin{aligned}
  s, w_i \in \RR^+ \\
  y_{i,j}\in \RR\\
  k\in \On
\end{aligned}
\right\}.
\]%
This is known as Iwasawa decomposition. 
Moreover, the Haar measure of $E\subseteq \GL_n(\RR)$ can be computed in terms of the Euclidean coordinates in this decomposition as%
\[
  \lambda_{\GL_n(\RR)}(E)
  =
  \int\limits_{\mathclap{ \RR^+\times\On\times \Dn \times \Tn}}
  \mathds{1}_E(skwy) s^{-1}
  \prod_{i=1}^{n-1} w_i^{2(n-i)-1} \; dk \; ds \; \prod_{i=1}^{n-1}dw_i \; \prod_{i< j} dy_{i,j},
\]
where $dk$ is the normalized Haar measure on $\On$, and $ds, dw_i, dy_{i,j}$ are Lebesgue measure on $\RR$. 
\end{theorem}
\begin{proof}
We include a short proof to be self-contained, even though proofs for similar decompositions can be found in Lie theory literature.
 
 In this proof, we denote the Haar measure of a group $G$ by $\mu_G$. First note that $\det:\GL_n(\RR)\rightarrow \RR^*$ is a group homomorphism, where $\RR^*$ is the multiplicative group of nonzero real numbers. Clearly, $H:=\det^{-1}(\{1,-1\})$ is a closed normal subgroup of $\GL_n(\RR)$, and $\RR^+$ is isomorphic with $\GL_n(\RR)/H$. So by Theorem~2.51 of \cite{Fol}, the Haar measure of $\GL_n(\RR)$ can be decomposed as $d\mu_{\GL_n}(sx)=d\mu_{\RR^+}(s) \, d\mu_{H}(x)$. Next, consider the Iwasawa decomposition 
 $\SL_n(\RR)=\SOn\Dn\Tn$, and note that by taking inverse and allowing matrices with determinant of -1, we can write $H=\Tn\Dn\On$. Since $\On$ and $H$ are unimodular groups, we can apply Theorem~2.51 of \cite{Fol} again, to get the decomposition 
 $d\mu_{H}(ywk)=d\mu_{\On}(k) \, d\mu_{\Tn\Dn}(yw)$. 
 
 Finally, we need to compute $d\mu_{\Tn\Dn}$.  To do so, note that $(yw)(y'w')=(y(wy'w^{-1}))\, (ww')$. Since $\Dn$ normalizes $\Tn$ (i.e. $w^{-1}yw\in \Tn$, whenever $w\in \Dn$ and $y\in \Tn$), we can view $\Tn\Dn$ as a semidirect product $\Tn\rtimes \Dn$, with $w\in \Dn$ acting on $y\in \Tn$ by $w\cdot y=wyw^{-1}$. Therefore, by standard results in semidirect products of groups (e.g. see Section 1.2 of \cite{TaylorKaniuth:book}), we have $d\mu_{\Tn\Dn}(yw)=\delta(w)^{-1}d\mu_{\Tn}(y)\, d\mu_{\Dn}(w)$, where $\delta:\Dn\rightarrow \RR^+$ satisfies 
 $$\int_{\Tn} f(y)d\mu_{\Tn}(y)=\delta(w)\int_{\Tn} f(wyw^{-1}) dy, \quad \mbox{ for every } f\in C_c(\Tn) \mbox{ and } w\in \Dn.$$
One can easily see that $d\mu_{\Dn}(w)=\prod_{i=1}^{n-1}\frac{dw_i}{w_i}$ and $d\mu_{\Tn} (y)=\prod_{i< j} dy_{i,j}$, with matrices $w$ and $y$ as represented in the statement of the theorem. Given that $wyw^{-1}=[y'_{i,j}]\in \Tn$ with 
$y'_{i,j}=\frac{w_iy_{ij}}{w_j}$, we have $\delta(w)=\prod_{1\leq i< j\leq n}\frac{w_j}{w_i}$.  We can compute $\delta$ by a simple counting argument, and using the fact that $w_n=(w_1\cdots w_{n-1})^{-1}$, as follows.
\begin{eqnarray*}
\delta(w)=\prod_{1\leq i< j\leq n}\frac{w_j}{w_i}=\frac{\prod_{1\leq i\leq n}w_i^{i-1}}{\prod_{1\leq i\leq n}w_i^{n-i}}=\prod_{1\leq i\leq n}{w_i^{2i-n-1}}=\prod_{1\leq i\leq n-1}{w_i^{2i-2n}}.
\end{eqnarray*} 
 
 Putting all these together, we obtain the Haar measure of $\GL_n(\RR)$, when the decomposition $\GL_n(\RR)=\RR^+\Tn\Dn\On$ is used:
\[
  \lambda_{\GL_n(\RR)}(E)
  =
  \int\limits_{\mathclap{ \RR^+\times\Tn\times \Dn \times \On}}
  \mathds{1}_E(sywk) 
  \prod_{i=1}^{n-1} w_i^{2i-2n} \; dk \; \frac{ds}{s} \; \prod_{i=1}^{n-1}\frac{dw_i}{w_i} \; \prod_{i< j} dy_{i,j}.
\]
Finally, applying inverse map to the above formula and using unimodularity of $\GL_n(\RR)$, we get
\begin{eqnarray*}
  \lambda_{\GL_n(\RR)}(E)
  &=&
  \int\limits_{\mathclap{ \RR^+\times\On\times \Dn \times \Tn}}
  \mathds{1}_E(s^{-1}k^{-1}w^{-1}y^{-1}) 
  \prod_{i=1}^{n-1} w_i^{2i-2n} \; d{k} \; \frac{d{s}}{s} \; \prod_{i=1}^{n-1}\frac{d{w_i}}{w_i} \; \prod_{i< j} dy_{i,j}\\
  &=&
 \int\limits_{\mathclap{ \RR^+\times\On\times \Dn \times \Tn}}
  \mathds{1}_E(skwy) 
  \prod_{i=1}^{n-1} w_i^{-2i+2n} \;  d{k^{-1}} \; \frac{ds^{-1}}{s^{-1}} \; \prod_{i=1}^{n-1}\frac{d{w_i^{-1}}}{w_i^{-1}} \; \prod_{i< j} d(-y_{i,j}) \\
  &=&
  \int\limits_{\mathclap{ \RR^+\times\On\times \Dn \times \Tn}}
  \mathds{1}_E(skwy) 
  \prod_{i=1}^{n-1} w_i^{-2i+2n} \;  d{k} \; \frac{ds}{s} \; \prod_{i=1}^{n-1}\frac{d{w_i}}{w_i} \; \prod_{i< j} dy_{i,j}.
\end{eqnarray*}
\end{proof}

We now extend the tiling system originally constructed in \cite{GST} for $\GL_2(\RR)$ to $\GL_n(\RR)$.

Let $F$ be the set%
\begin{gather*}
F = \left\{
s
k
\begin{pmatrix}
w_1 & w_1 y_{1,2} & w_1 y_{1,3} & \cdots & w_1 y_{1,n-1} & w_1 y_{1,n} \\
      & w_2 & w_2 y_{2,3} & \cdots & w_2 y_{2,n-1} & w_2 y_{2,n} \\
               && w_3 & \cdots & w_3 y_{3,n-1} & w_3 y_{3,n} \\
                          &&& \ddots & & \vdots \\
                                         &&&& w_{n-1} & w_{n-1} y_{n-1,n} \\
                                            &&&&& \prod_{i=1}^{n-1}w_i^{-1}
\end{pmatrix} \,\, \middle| \,\, 
\begin{aligned}
  s, w_i \in [1,2)\\
  y_{i,j}\in [0,1)\
  k \in \On
\end{aligned}
\right\}
\end{gather*}
and let $P$ be the discrete set%
\[
P = 
\left\{
2^\lambda
\begin{pmatrix}
2^{\kappa_1} & 2^{\kappa_2}\mu_{1,2} & 2^{\kappa_3}\mu_{1,3} & \cdots &2^{\kappa_{n-1}}\mu_{1,n-1} & 2^{\kappa_n}\mu_{1,n} \\
& 2^{\kappa_2} & 2^{\kappa_3}\mu_{2,3} & \cdots &2^{\kappa_{n-1}}\mu_{2,n-1} & 2^{\kappa_n}\mu_{2,n} \\
& & 2^{\kappa_3} & \cdots&2^{\kappa_{n-1}}\mu_{3,n-1}  & 2^{\kappa_n} \mu_{3,n} \\
& & & \ddots &  & \vdots \\
& & & & 2^{\kappa_{n-1}} &2^{\kappa_{n}}\mu_{n-1,n} \\
& & & & & 2^{\kappa_n}  
\end{pmatrix} \,\,\middle|\,\, \lambda, \kappa_i, \mu_{i,j}\in \mathbb{Z} \right\}
\]
where $\kappa_n = -\sum_{i=1}^{n-1} \kappa_i$.

\begin{proposition}\label{prop:OriginalTilingDefinitionSatisfied}
For $F$ and $P$ as above, the following two properties hold:
\begin{enumerate}
   \item $F\cdot p \cap F \cdot q = \emptyset$ for every $p\neq q$ in $P$.
   \item $\bigcup_{p\in P} F \cdot p = \GL_n(\RR)$.
\end{enumerate}
\end{proposition}

\begin{proof}
Again, for simplicity of notation, let 
$w_n=\prod_{i=1}^{n-1} w_i^{-1}$. Now,
to establish the claim, we show that 
for each $a \in \GL_n(\RR)$, the equation%
\begin{equation}\label{eq1}
a = 
s
k
\begin{pmatrix}
w_1 & \cdots & w_1 y_{1,n-1} & w_1 y_{1,n} \\
                          & \ddots & & \vdots \\
                                         && w_{n-1} & w_{n-1} y_{n-1,n} \\
                                            &&& w_n
\end{pmatrix}
2^\lambda
\begin{pmatrix}
2^{\kappa_1} & \cdots &2^{\kappa_{n-1}}\mu_{1,n-1} & 2^{\kappa_n}\mu_{1,n} \\
& \ddots &  & \vdots \\
& & 2^{\kappa_{n-1}} &2^{\kappa_{n}}\mu_{n-1,n} \\
& & & 2^{\kappa_n}  
\end{pmatrix}
\end{equation}
has a unique solution subject to the constraints
$\lambda, \kappa_i, \mu_{i,j}\in \mathbb Z, s, w_i\in [1,2), y_{i,j}\in [0,1)$.

First note that by uniqueness of Iwasawa decomposition, the element $k\in\On$ in the above equation is unique. Moreover, from (\ref{eq1}) we have $|\det(a)|=(s2^\lambda)^n$, and it is easy to see that there are unique $s\in [1,2)$ and $\lambda\in {\mathbb Z}$ which satisfy this equation. Indeed, this is a consequence of the fact that $([1,2), {\mathbb Z})$ is a tiling system for $(0,\infty)$.

Let $a'=\frac{1}{|\det(a)|^{1/n}}k^{-1}a$. (Note that at this point $k$ is uniquely determined.) Clearly, $a'$ is upper triangular with determinant 1, and we are left to show that%
\begin{equation}\label{eq2}
a' = 
\begin{pmatrix}
w_1 & \cdots & w_1 y_{1,n-1} & w_1 y_{1,n} \\
                          & \ddots & & \vdots \\
                                         && w_{n-1} & w_{n-1} y_{n-1,n} \\
                                            &&& w_n
\end{pmatrix}
\begin{pmatrix}
2^{\kappa_1} & \cdots &2^{\kappa_{n-1}}\mu_{1,n-1} & 2^{\kappa_n}\mu_{1,n} \\
& \ddots &  & \vdots \\
& & 2^{\kappa_{n-1}} &2^{\kappa_{n}}\mu_{n-1,n} \\
& & & 2^{\kappa_n}  
\end{pmatrix}
\end{equation}
has a unique solution given the constraints $w_i\in [1,2), y_{i,j}\in [0,1), \kappa_i, \mu_{i,j}\in \ZZ$.
To do so, we denote the product of the
two matrices on the right hand side of (\ref{eq2}) as $z=[z_{i,j}]$, and we observe
that these entries have the form%
\[
z_{i,j} = \begin{cases}\displaystyle
  0, & \textnormal{for } j-i<0, \\
  w_i2^{\kappa_i},& \textnormal{for } j-i=0, \\
  w_i2^{\kappa_j} (y_{i,j}+\mu_{i,j}), & \textnormal{for } j-i=1, \\
  w_i 2^{\kappa_j}(y_{i,j}+\mu_{i,j} 
    + \sum_{k=i+1}^{j-1} y_{i,k}\mu_{k,j}) & \textnormal{for } j-i>1
\end{cases}
\]

We now proceed by an induction-like argument on the diagonals, starting from the main diagonal.
Namely, note that $a_{i,i}'=2^{\kappa_i}w_i$, with constraints $w_i\in [1,2)$ and $\kappa_i\in{\mathbb Z}$, has a unique solution  for each $1\leq i \leq n$.
Moving to the super-diagonal, we likewise solve these equations to find that when $j-i=1$, we have%
\[ \mu_{i,j} = \left\lfloor \frac{2^{-\kappa_j}a_{i,j}'}{w_i} \right\rfloor, 
\quad \textnormal{and }\quad 
y_{i,j} = \frac{2^{-\kappa_j}a_{i,j}'}{w_i} - \left\lfloor \frac{2^{-\kappa_j}a_{i,j}'}{w_i} \right\rfloor. \]%
Finally, the remaining entries when $j-i>1$ can be computed similarly, as follows.%
\[
\mu_{i,j} = 
   \left\lfloor \frac{2^{-\kappa_j}a_{i,j}'}{w_i}-
     \sum_{k=i+1}^{j-1} y_{i,k}\mu_{k,j} 
   \right \rfloor
   , \quad \textnormal{and }\quad 
\]
\[ y_{i,j}  =
   \frac{2^{-\kappa_j}a_{i,j}'}{w_i}- \sum_{k=i+1}^{j-1} y_{i,k}\mu_{k,j}
   - \left\lfloor 
     \frac{2^{-\kappa_j}a_{i,j}'}{w_i}-
     \sum_{k=i+1}^{j-1} y_{i,k}\mu_{k,j} 
   \right\rfloor
\] where all values on the right hand sides are known from previous diagonals. This establishes the proposition.%
\end{proof}

Note that this proves that the pair $(P,\overline{F})$ forms a frame generator in the sense of \cite{BeT}, and the pair $(F,P)$ satisfies the slightly different definition of a tiling system given in \cite{GST}.
We now show that $(F,P)$ also fulfills the new conditions for a tiling system given in Definition \ref{def:TilingSystem}.

\begin{corollary}
For $F$ and $P$ as given in the previous proposition, we have
\begin{enumerate}
\renewcommand{\labelenumi}{{\rm(\roman{enumi})}}
\renewcommand{\theenumi}{{\rm(\roman{enumi})}}
  \item $\lambda_{\GL_n}(\overline F \cdot p \cap \overline F \cdot q) =0$ for every distinct pair $p,q\in P$,
  \item $\lambda_{\GL_n}(\GL_n(\RR)\setminus \bigcup\{\overline F\cdot p \; : \; p\in P\}) =0$.
\end{enumerate}
\end{corollary}

\begin{proof}
Property (ii) is immediate, as $\GL_n(\RR)=\bigcup\{F\cdot p \;:\; p\in P\}\subseteq \bigcup \{\overline F\cdot p \;:\;p\in P\} \subseteq \GL_n(\RR)$. For the first property, note that by Proposition \ref{prop:OriginalTilingDefinitionSatisfied}, $F\cdot p\cap F\cdot q=\emptyset$ if $p\neq q \in P$. So,
\[ \begin{split}
 \lambda_{\GL_n(\RR)}(\overline{F}\cdot p \cap \overline F \cdot q) 
&= \lambda_{\GL_n(\RR)}((\overline{F}\setminus F)\cdot p \cap \overline F\cdot q) \cup(\overline{F}\setminus F)\cdot q \cap \overline F\cdot p) )\\ 
& \leq  \lambda_{\GL_n(\RR)}((\overline{F}\setminus F)\cdot p)+ \lambda_{\GL_n(\RR)}((\overline{F}\setminus F)\cdot q)\\
& = \lambda_{\GL_n{(\RR)}}(\overline F \setminus F) + \lambda_{\GL_n{(\RR)}}(\overline F \setminus F) \\
& = 0,
\end{split} \]
where the second to last equality follows as $\GL_n(\RR)$ is unimodular, and the last equality can easily be computed directly by Theorem \ref{thm:IwasawaHaarMeas}.
\end{proof}

\paragraph{An Open Set $\Fo\supset \overline F$.}
Fix $\epsilon>0$.
Let $\Fo$ be the set%
\[
\Fo = \left\{
s
k
\begin{pmatrix}
w_1 & w_1 y_{1,2} & w_1 y_{1,3} & \cdots & w_1 y_{1,n-1} & w_1 y_{1,n} \\
      & w_2 & w_2 y_{2,3} & \cdots & w_2 y_{2,n-1} & w_2 y_{2,n} \\
               && w_3 & \cdots & w_3 y_{3,n-1} & w_3 y_{3,n} \\
                          &&& \ddots & & \vdots \\
                                         &&&& w_{n-1} & w_{n-1} y_{n-1,n} \\
                                            &&&&& \prod_{i=1}^{n-1}w_i^{-1}
\end{pmatrix} \,\, \middle| \,\, 
\begin{aligned}
  s, w_i \in (1-\epsilon,2+\epsilon), \\
  y_{i,j}\in (-\epsilon,1+\epsilon), \\
  k \in \On
\end{aligned}
\right\}.
\]

\begin{claim}
$\Fo$ is an open set in $\GL_n(\RR)$ such that $\overline F \subseteq \Fo$.
\end{claim}

\begin{proof}
The inclusion $\overline F\subseteq \Fo$ is clear.
To prove that $\Fo$ is open in $\GL_n(\RR)$, note that the multiplication map 
$$ \SOn\times\Dn\times\Tn\rightarrow \SL_n(\RR)$$
is a diffeomorphism (see Proposition 1.6.2 of \cite{Basic-Lie:Abbaspour-Moskowitz}). So, $\RR^+\times \On\times\Dn\times\Tn\rightarrow \GL_n(\RR)$ is a homeomorphism, and maps any open set to an open subset of $\GL_n(\RR)$. 
\end{proof}

Summarizing the results from this section, we have constructed a tiling system $(F,P)$ and determined how to compute the measure of a set $E$ through the coordinates of the Iwasawa decomposition, which we used to motivate the tiling. We now use these results to compute frame bounds for the discretization of the continuous wavelet transform.

\section{Computation of Frame Bounds}
\label{sec:ComputationAndExample}
In this section, we find bounds for the value $M$ (as defined for Theorem \ref{theorem:frame}) associated with  the sets $P,F,$ and $\Fo$.
Recall that $M$ is just the least uniform upper bound for the number
of $q\in P$ such that for a fixed $p\in P$, $\overline F\cdot p$
and $\Fo\cdot q$ have nontrivial intersection. 
Once we have
done this for general $n$, we provide a concrete example
with all details when $n=2$.
\begin{notation}
From this point forward, we use ${\rm diag}(w_1,\ldots,w_n)$ to denote an $n\times n$ diagonal matrix with diagonal entries $w_1,\ldots,w_n$.
\end{notation}
\begin{proposition}\label{prop:boundOnM}
For $F$, $P$, and $\Fo$ as defined before,  $M$ as in Notation \ref{notation:Ip}, and  $0<\epsilon\leq \frac{1}{2}$, we have
\[ M =\sup_{p\in P}|\{k\in P:\ \overline{F}\cdot p \cap \Fo\cdot k\neq \emptyset\}|  \leq 3^n 6^{\frac{n(n-1)}{2}}.\]
\end{proposition}

\begin{proof}
If $\overline F \cdot p \cap \Fo \cdot p'\neq \emptyset$ for some $p\neq p'\in P$, then there exist $a\in \overline F$ and $a'\in \Fo$ for which we have $ap=a'p'$, or equivalently%
\begin{equation}\label{eq-first}
a=a'p'p^{-1}.
\end{equation}
Using Theorem \ref{thm:IwasawaHaarMeas}, we write the Iwasawa decompositions $a=skwy$ and $a'=s'k'w'y'$, where $w$ and $w'$ are diagonal matrices ${\rm diag}(w_1,\ldots, w_{n-1}, (w_1\cdots w_{n-1})^{-1})$ and 
${\rm diag}(w'_1,\ldots, w'_{n-1}, (w'_1\cdots w'_{n-1})^{-1})$, and $y=[y_{i,j}]$ and $y'=[y'_{i,j}]$ are unit upper triangular matrices. Suppose $p,p'$ are written in the format of $P$ using $\lambda, \kappa_i, \mu_{i,j}$ and $\lambda', \kappa'_i, \mu'_{i,j}$ respectively.
By the uniqueness of Iwasawa decomposition and equality of the determinants of both sides of the above equation, we get:
\begin{align}
 \label{C1} \tag{C1} k &= k', \\ 
 \label{C2}\tag{C2} s &= s' 2^{\lambda'-\lambda}.
\end{align}
From  Equation (\ref{eq-first}), we get 
\begin{equation}\label{eq-2nd}
wy=w'y' \, [\mu'_{i,j}]\, {\rm diag}(2^{\kappa'_1-\kappa_1},\ldots, 2^{\kappa'_n-\kappa_n}) \, [\mu_{i,j}]^{-1},
\end{equation}
which simplifies to 
\begin{equation*}
wy= \underbrace{w'\, {\rm diag}(2^{\kappa'_1-\kappa_1},\ldots, 2^{\kappa'_n-\kappa_n})}_{\in \Dn}\, \underbrace{{\rm diag}(2^{-\kappa'_1+\kappa_1},\ldots, 2^{-\kappa'_n+\kappa_n})\, y' \, [\mu'_{i,j}]\, {\rm diag}(2^{\kappa'_1-\kappa_1},\ldots, 2^{\kappa'_n-\kappa_n}) \, [\mu_{i,j}]^{-1}}_{\in \Tn},\nonumber
\end{equation*}
where $y,y',[\mu_{i,j}],[\mu'_{i,j}]\in \Tn$ are unit upper triangular matrices , and $\sum_{i=1}^n\kappa_i=\sum_{i=1}^n\kappa'_i=0$ (as in the definition of $P$). Note that in the above equation, we used the fact that $\Dn$ normalizes $\Tn$. 
By uniqueness of Iwasawa decomposition, we have 
$$w=w'\, {\rm diag}(2^{\kappa'_1-\kappa_1},\ldots, 2^{\kappa'_n-\kappa_n}),$$
{ and }
$$[y_{i,j}]={\rm diag}(2^{-\kappa'_1+\kappa_1},\ldots, 2^{-\kappa'_n+\kappa_n})\, [y'_{i,j}] \, [\mu'_{i,j}]\, {\rm diag}(2^{\kappa'_1-\kappa_1},\ldots, 2^{\kappa'_n-\kappa_n}) \, [\mu_{i,j}]^{-1}.$$
Thus we obtain the conditions
\begin{align}
 \label{C3}\tag{C3}w_i'2^{\kappa'_i-\kappa_i} &= w_i & \quad \quad \forall i \\
\label{C4}\tag{C4}2^{(\kappa_j'-\kappa_j)-(\kappa_i'-\kappa_i)}\sum_{k=i}^j y_{i,k}'\mu_{k,j}'&=\sum_{k=i}^jy_{i,k}\mu_{k,j}  & \quad \quad \textnormal{for }j\geq i
\end{align}
by comparing each product entry-wise.

Given a fixed $p\in P$ represented by $\lambda, \kappa_i, \mu_{i,j}$, we count the number of $p'\in P$ with parameters $\lambda', \kappa_i', \mu_{i,j}'$ for which Equation (\ref{eq-2nd}) can be satisfied for some choice of $a\in\overline{F}$ and $a'\in \Fo$. Using condition (\ref{C1}), we see that%
\[
  s' 2^{\lambda'-\lambda} = s \in [1,2].
\]%
As $s'\in(1-\epsilon,2+\epsilon)$, we deduce that $\lambda'-\lambda\in \{-1,0,1\}$ if $\epsilon\leq \frac{1}{2}$. This also shows that $\lambda'-\lambda=-1$ implies $s'\in [2,2+\epsilon)$, $\lambda'-\lambda=0$ implies $s'\in[1,2]$, and finally that $\lambda'-\lambda=1$ implies that $s'\in (1-\epsilon,1]$.
Similarly, we deduce from (\ref{C3}) that $\kappa_i'-\kappa_i\in \{-1,0,1\}$ holds when $\epsilon\leq \frac{1}{2}$. This constrains $w_i'$ to be in $[2,2+\epsilon)$, $[1,2]$ or $(1-\epsilon,1]$ respectively. Also, note that for a given choice of $\kappa_i'\in\{\kappa_i-1,\kappa_i,\kappa_i+1\}$, we have a uniquely determined $w'_i$ given by $w_i'=\frac{w_i}{2^{\kappa'_i-\kappa_i}}$.

From this point forward, we assume that $\lambda'$ and $\kappa'_1,\ldots,\kappa_{n-1}'$ have been chosen according to the above constraints, and the parameters $s'$, $w_i'$ and $\kappa'_n$ have been determined accordingly; recall that $\kappa'_n=-\sum_{i=1}^{n-1}\kappa'_i$.
Let $p_{i,j}=(\kappa_j'-\kappa_j)-(\kappa_i'-\kappa_i)$.  As $\kappa_j'-\kappa_j$, $\kappa_i'-\kappa_i \in \{-1,0,1\}$, it follows that $p_{i,j}\in \{-2,-1,0,1,2\}$.
To count possible solutions for equations, we will make repeated use of the following claim. 
\begin{claim}\label{claim:count}
Let $\epsilon\leq \frac{1}{2}$ be fixed. Then for any interval $[\alpha, \alpha+4]\subset \mathbb R$, there are at most six $\beta\in \mathbb Z$ such that $[\alpha, \alpha+4]\cap \left(\beta-\epsilon, \beta+1+\epsilon\right)\neq \emptyset$. 
\end{claim}
\begin{proof}[Proof of claim.]
Let $m=\min \{\beta\in \mathbb Z\colon [\alpha, \alpha+4]\cap (\beta-\epsilon, \beta+1+\epsilon)\neq \emptyset\}$. Clearly $m$ exists and is finite, as both intervals are bounded and $\mathbb Z$ is discrete. Now, consider $m+k$ for $k\geq 6$. If $(m+k-\epsilon, m+k+1+\epsilon)$ intersects $[\alpha, \alpha+4]$, then we have $(m-\epsilon, m+1+\epsilon)$ intersects $[\alpha-k, \alpha-k+4]$. On the other hand,  by definition of $m$, we know that $(m-\epsilon, m+1+\epsilon)$ intersects $[\alpha, \alpha+4]$ as well. This is a contradiction, because $(m-\epsilon, m+1+\epsilon)$ has length at most $2$, which is not larger than the gap between the above two closed intervals. Therefore, there are at most 6 possible choices for $\beta\in\mathbb Z$ for which we may have  $[\alpha, \alpha+4]\cap \left(\beta-\epsilon, \beta+1+\epsilon\right)\neq \emptyset$.
\end{proof}

We now proceed by an inductive argument, examining diagonals of the matrices on the two sides of Equation (\ref{C4}), starting from the super-diagonal.
For $j=i+1$, condition (\ref{C4}) becomes
\[y_{i,i+1}'+\mu_{i,i+1}'=2^{-p_{i,i+1}}(\mu_{i,i+1}+y_{i,i+1}),\]
with the constraints $y_{i,i+1}'\in (-\epsilon,1+\epsilon)$, $y_{i,i+1}\in [0,1]$ and $\mu_{i,i+1}'\in {\mathbb Z}$. 
Note that $2^{-p_{i,i+1}}\mu_{i,i+1}$ is fixed at this point, and clearly $2^{-p_{i,i+1}}y_{i,i+1}\in [0,4]$.
So by Claim \ref{claim:count}, there are at most six possible choices for $\mu_{i,i+1}'$. This settles the discussion for the superdiagonal. 

Next, by induction hypothesis, assume that for every  $\mu'_{i,j}$ with $j-i<k$, there are at most six possible choices that satisfy Equation (\ref{C4}). Also assume  that  $y_{i,j},y'_{i,j}$ and $\mu'_{i,j}$ have been fixed whenever $j-i<k$ (i.e. for  the first $k-1$ diagonals).
For $j=i+k$,  condition (\ref{C4}) becomes
\begin{equation}\label{eq-final}
y_{i,i+k}'+\mu_{i,i+k}'=2^{-p_{i,i+k}}y_{i,i+k}+\underbrace{\left(2^{-p_{i,i+k}}\sum_{t=i}^{k+i-1}y_{i,t}\mu_{t,{i+k}}-\sum_{t=i+1}^{k+i-1}y'_{i,t}\mu'_{t,{i+k}}\right)}_{\alpha},
\end{equation}
where $\mu_{i,i+k}'\in {\mathbb Z}$, $y_{i,i+k}'\in (-\epsilon,1+\epsilon)$ and $y_{i,i+k}\in [0,1]$. Note that the expression $\alpha$ in the above equation  is fixed, as it only involves values which have already been chosen. So, by Claim \ref{claim:count}, there are at most six possible choices for 
$\mu_{i,i+k}'$. This proves that for every $\mu'_{i,j}$ in Equation (\ref{C4}), there are only six possible values that may satisfy the equation. 

In the following table, we summarize what we have found so far. Note that not every possible solution is necessarily an actual solution; however, any solution for (\ref{C4}) is counted below. 

\begin{table}[!htb]
\centering
{\begin{tabular}{|c|c|c|}
\hline
\mbox{parameter}& \mbox{number of possible choices}& possible solutions in terms of $p$\\
\hline
$\lambda'$&3& $\lambda-1,\lambda, \lambda+1$\\
\hline
$\kappa_i', 1\leq i<n$ & 3 & $\kappa_i-1, \kappa_i,\kappa_i+1$\\
\hline
$\kappa_n'$ & 1& $-\sum_{i=1}^{n-1}\kappa_i'$\\
\hline
$\mu'_{i,j}, i<j$ & 6& solutions to Equation (\ref{eq-final}) satisfying given constraints\\
\hline
\end{tabular}}
\caption{Maximum number of possible choices for each parameter in $p'$.}\label{table:Intersections}
\end{table}

Combining these results, we conclude that for a fixed $p\in P$, there are at most $3^n 6^{\frac{n(n-1)}{2}}$ possible choices for $p'\in P$ such that $\overline F \cdot p \cap \Fo \cdot p'\neq \emptyset$.
\end{proof}
\begin{remark}

When $n=2$, the above proposition gives $M\leq 54$.
However, note that in this case $p' = 2^{\lambda'} \begin{bmatrix}
2^{\kappa'} & 2^{-\kappa'}\mu' \\
0 & 2^{-\kappa'}
\end{bmatrix}$, therefore we only have one $\kappa'$ and $\mu'$ to choose, and $p_{1,2}\in \{-2, 0, 2\}$.
Therefore by bounding $\mu'$ for each $p_{1,2}$ separately, we can obtain an improved bound.
By Claim \ref{claim:count}, we still have at most 6 choices when $p_{1,2}=-2$.
Adapting Claim \ref{claim:count} in the case $p_{1,2}=0$, (intervals $[\alpha, \alpha+1]$), there are at most 3 choices for $\mu'$.
Finally, for in the case $p_{1,2}=2$ (intervals $[\alpha,\alpha+\frac{1}{4}]$), we have for $0<\epsilon<\frac{1}{4}$ there are at most 2 choices of $\mu'$, and for $\frac{1}{4}\leq \epsilon<\frac{1}{2}$ there are at most 3 choices for $\mu'$.
So based on $\epsilon$, one gets that $M\leq 33$ and $M\leq 36$, respectively.

The case for $n=2$ was also previously studied in \cite{GST}.
Proposition \ref{prop:boundOnM}, together with Theorem \ref{theorem:frame}, shows that using our methods one can construct frames with significantly better frame condition numbers than the construction in  \cite{GST}. 
Indeed, the ratio of the frame bounds in our construction is $\frac{C_2}{C_1}\leq 54$, whereas the frame condition number for the construction in \cite{GST} was $\frac{C_2}{C_1}\sim 1782$. (In fact, it is only mentioned in \cite{GST} that $M<\infty$. However, looking into their arguments closely, one can obtain the bound $M\sim 1782$. Also, note that there is a typo in the definition of $M$ in \cite{GST}; the correct formula should be 
$M=\sup_{p\in P}|\{p'\in P: p\cdot D\cap p'\cdot D\neq \emptyset\}|.$) As the rate of convergence in frame calculations when approximating signals is highly sensitive to the frame condition number (i.e. the ratio of the frame bounds), our methods result in much more practical and efficient frames than those of \cite{GST} for the case of $n=2$. 
\end{remark}

\section{Concrete Frame Construction for \texorpdfstring{$n=2$}{n=2}}
So far, we have established that the proposed sets $F, \Fo$, and $P$ satisfy the conditions of Theorem \ref{theorem:frame}, and have calculated $M$ as well.  In this section, we use these sets and follow the construction in Theorem \ref{theorem:frame}, to give an explicit example of a discrete frame for $L^2(\M_2(\RR))$. A similar approach can be taken for higher dimensions, but we focus our attention on $n=2$ for now.

The Iwasawa decomposition for $\GL_2(\RR)$ can be stated as follows: Let ${\rm O}_2$ denote the group of orthogonal $2\times 2$ matrices, ${\rm D}_2$ denote the diagonal $2\times 2$ matrices with positive diagonal entries and determinant 1, and ${\rm T}_2$ denote the $2\times 2$ unit upper triangular matrices.
Every element of $\GL_2(\RR)$ can be uniquely decomposed as an ordered product
of elements in ${\rm O}_2$, ${\rm D}_2$, and ${\rm T}_2$. That is, $\GL_2(\RR)={\rm O}_2 {\rm D}_2{\rm T}_2$. Note that ${\rm O}_2$ is compact, and ${\rm T}_2$ and ${\rm D}_2$ are both abelian subgroups of $\GL_2(\RR)$.

\paragraph{Tiling System.} 
Let $P=\left\{2^\lambda \left(
\begin{array}{cc}
2^{\kappa} &  2^{-\kappa}\mu  \\
0   &   2^{-\kappa}
\end{array}
\right): \lambda, \kappa,\mu \in {\mathbb Z}\right\}$,
\[F=\left\{
\left(
\begin{array}{cc}
  \pm\cos \theta  & \mp\sin \theta \\
  \sin\theta  & \cos\theta  \\
\end{array}
\right)
\left(
\begin{array}{cc}
s w &  s w y \\
0   &  s w^{-1}
\end{array}
\right) : \theta \in [0,2\pi),\,
s,w  \in [1,2),\, y\in[0,1) \right\}.\]
Then $(F,P)$ forms a tiling system in the sense of Definition \ref{def:TilingSystem} for $\GL_2(\RR)$. After projecting on the $(s,w,y)$-space, this tile and some of its translates under the action of $P$ are shown in Figure \ref{fig:Tiling}. Note that by  Proposition \ref{prop:boundOnM}, we have $M\leq 54$.

\begin{figure}[t]
  {\centering
  \includegraphics[clip,height=10cm]{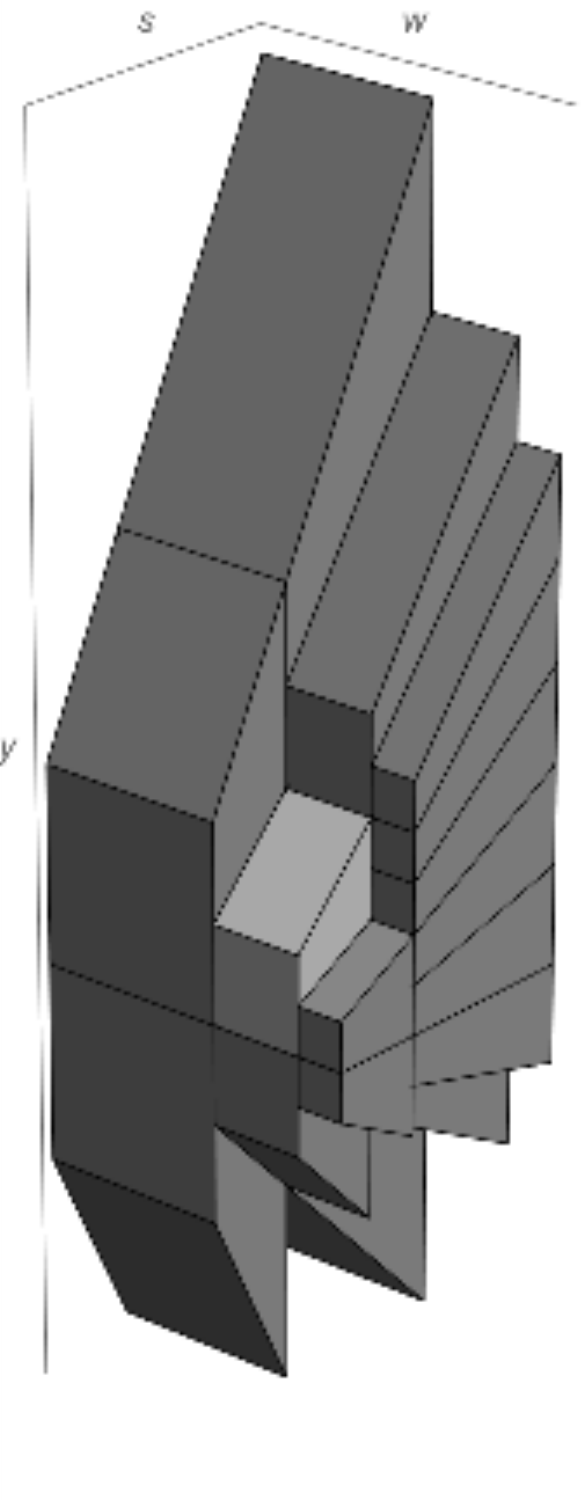}\par}%
  \caption{Multiscale tiling based on $(F,P)$. The base tile is lighter. Several shifted tiles are shown in darker gray.}%
  \label{fig:Tiling}%
\end{figure}

As before, we take $0<\epsilon\leq \frac{1}{2}$, and define $\Fo$ to be 
\[
\Fo= \left\{ \begin{pmatrix}
\pm\cos \theta & \mp\sin \theta\\
\sin \theta & \cos \theta
\end{pmatrix}
\begin{pmatrix}
s w & s w y \\
0 & s w ^{-1}
\end{pmatrix}\;\middle|\;
\begin{aligned}
\theta \in [0,2 \pi )\\
s, w \in (1- \epsilon , 2 + \epsilon) \\
y\in(-\epsilon, 1+\epsilon)
\end{aligned}
\right\}.
\]

\paragraph{Reconstruction Formula.}
Now, we recall the explicit formulation of the reconstruction formula (\ref{eqn:General Reconstruction Formula})  for $n=2$. The wavelet condition (Equation (\ref{A})) and reconstruction formula (Equation (\ref{B})) were obtained in Theorem 2.1 of \cite{GST}.

Let $\psi\in L^2(\mathbb{R}^4)$. If
\begin{equation}\label{A}
\int\limits_{\mathbb{R}^4}\left|\widehat{\psi}(h_{1},h_{2},h_{3},h_{4})\right|^2
\frac{dh_{1}dh_{2}dh_{3}dh_{4}}{|h_{1}h_{4}-h_{2}h_{3}|^2}=1,
\end{equation}
then $\psi$ is a wavelet.
For $x,y\in \M_2(\RR)$ and $h\in \GL_2(\RR)$, define 

%

\noindent
\resizebox{\textwidth}{!}{
$\psi_{x,h}(y)={\frac{1}{|h_{1}h_{4}-h_{2}h_{3}|}\psi\Big(\frac{h_4(y_1-x_1)-h_2(y_3-x_3), h_4(y_2-x_2)-h_2(y_4-x_4), h_1(y_3-x_3)-h_3(y_1-x_1), h_1(y_4-x_4)-h_3(y_2-x_2)}{h_1h_4-h_2h_3}\Big).}$}

Then, for any $f\in L^2(\mathbb{R}^4)$, we have
\begin{equation}\label{B}
f=\int\limits_{\mathbb{R}^4}\int\limits_{\mathbb{R}^4}\left\langle f,\psi_{x,h}\right\rangle  \psi_{x,h}\,\,
\frac{dx_1\cdots dx_4\,dh_1\cdots dh_4}{|h_{1}h_{4}-h_{2}h_{3}|^4},
\end{equation}
weakly in $L^2(\mathbb{R}^4)$.
Conversely, if {\rm (\ref{B})} holds for every $f\in L^2(\mathbb{R}^4)$, then
$\psi$ is a wavelet.

Next, we discretize this reconstruction formula to make it computationally feasible.

\paragraph{The Discrete Frame.} 
To discretize the above continuous frame, we find a 4-dimensional cube $R$ containing $\Fo$. Suppose 
$$R=\left\{(x_1,x_2,x_3,x_4): x_i\in[a_i,b_i]\ \mbox{ for }\ i\in \{1,2,3,4\}\right\},$$
where $a_i<b_i$, for $1\leq i\leq 4$, are fixed real numbers. We need to determine appropriate values for $a_i$ and $b_i$ so that $\Fo\subseteq R$. Consider an arbitrary element of $\Fo$ together with its Iwasawa decomposition, say
\[
\left(\begin{array}{cc}
x_1  & x_2   \\
x_3 & x_4  
\end{array}
\right)=\left(\begin{array}{cc}
\pm\cos\theta  & \mp\sin\theta   \\
\sin\theta & \cos\theta  
\end{array}
\right)\left(\begin{array}{cc}
sw & swy   \\
0 & \frac{s}{w}  
\end{array}
\right),
\]
where $s,w\in (1-\epsilon,2+\epsilon)$ and $y\in (-\epsilon, 1+\epsilon)$. Comparing the two sides of the above matrix equation, we get
\begin{eqnarray*}
|x_1|,|x_3|&\leq& |sw|\leq (2+\epsilon)^2< 7,\\
|x_2|,|x_4|&<& \sqrt{|swy|^2+|\frac{s}{w}|^2}\leq (2+\epsilon)(1+\epsilon)\sqrt{w^2+\frac{1}{w^2}}\leq(2.5)(1.5)\sqrt{6.5}<10,
\end{eqnarray*}
where we used the fact that $\epsilon\leq \frac{1}{2}$. Thus, we can set $a_1=a_3=-7$, $b_1=b_3=7$, $a_2=a_4=-10$, and $b_2=b_4=10$.

Let $L^2(R)$ be the closed subspace of
$L^2(\RR^4)$ consisting of all the elements supported on $R$.
We can construct an orthonormal basis of $L^2(R)$ indexed by the set $J$ defined as
\[
J=\left\{\lambda=\left(
\begin{array}{cc}
\lambda_1 &  \lambda_3  \\
\lambda_2  & \lambda_4
\end{array}
\right) \;\middle|\; \lambda_1,\lambda_3\in\frac{1}{14}\ZZ,\, \,  \lambda_2,\lambda_4\in\frac{1}{20}\ZZ
\right\}.
\]%
Then for all $g\in L^2(\M_2(\RR))$ which satisfy $\mathds{1}_{\overline F} \leq \widehat g \leq \mathds{1}_{\Fo}$, we have
\[ \{ \rho[\lambda,p]^{-1}g \; : \; (\lambda,p)\in J\times P\} \]
is a discrete frame with frame bounds $C_1=|R|, C_2=|R|M$. That is%
\[
  |R|\, \|f\|^2
  \leq
  \sum_{k\in P} \sum_{\gamma\in J} \left| \langle f, \rho[\gamma, k]^{-1}g \rangle_{L^2(\M_2(\RR))} \right|^2
  \leq
  |R|M\, \|f\|^2
\]
for all $f\in L^2(\M_2(\RR))$. If necessary or useful, it is easy to explicitly compute $\rho[\lambda,p]^{-1}g$, with a formula similar to $\psi_{x,h}=\rho[x,h]\psi$ (which was explicitly computed previously). 

\begin{remark}
Note that $|R|$ in the above construction is $14^2\times 20^2$, which is by far smaller than the similar parameter from \cite{GST}, which was $176^4$. 
\end{remark}

\section{%
\texorpdfstring{{\color{\ke}Conclusions and}}{Conclusions and} Future Directions}
\label{sec:Discussion}
{\color{\ke}
In this article, we construct a novel frame for the Hilbert space $L^2(\mathbb R^{n^2})$ for arbitrary $n$.
We do so by carefully considering the action of $\GL_n(\mathbb R)$ on $\M_n(\mathbb R)$ to develop a well-spread sampling set which is particularly well-structured, also making it computationally tractable for purposes of applications.
Additionally, this approach provides more freedom in the choice of analyzing wavelet, as opposed to the methods of coorbit theory.
More importantly, we view this work as a prototypical example demonstrating methods for finding explicit atomic decompositions/discrete frames, which can be applied to other classes of semidirect product groups.

Finally, our construction gives explicit frame bounds for general signals, which is not the case for arbitrary sampling sets (see Remark \ref{refremark} for more discussion).
We are currently investigating how restrictions to certain subclasses of signals can yield additional control over the frame bounds.
In the near future, we expect to also}
construct similar frames for the Hilbert space of Sobolev functions $H^k(\mathbb{R}^{n^2})$.
We expect the regularity conditions will provide additional control over the frame bounds.
This should result in a much smaller frame ratio, which we expect to be dependent on $\epsilon$; as opposed to the current construction for which the frame bounds are independent of parameter $\epsilon$.
In particular, we expect to see direct relations between how small the frame ratio is and how much regularity can be assumed on the signals of interest.

\section{Acknowledgements}
This project was initiated during a summer research program funded by University of Delaware \emph{Graduate Program Improvement and Innovation Grants} ``GEMS''. The second  author thanks University of Delaware for funding  his GEMS project in summer 2016 and summer 2018. The first author was partially supported by University of Delaware Research Foundation, and partially by NSF grant DMS-1902301, while this work was being completed. She also thanks the Department of Mathematical Sciences at Delaware for its support while important revisions were made to the paper.
{\color{\ke} We sincerely thank the anonymous reviewer for critically reading the manuscript and suggesting substantial improvements. Additionally we would like to thank Karlheinz Groechenig for several helpful remarks on an early draft and providing information about historical context of our work.}
Finally, the authors would like to thank Nathaniel Kim and Paige Shumskas for proofreading early drafts of this work.

\section*{References.}
\vspace{-0.6cm}
\renewcommand{\refname}{}

\end{document}